\documentclass[12pt]{amsart}
\usepackage{amscd,setspace,amssymb,amsopn,amsmath,amsthm,mathrsfs,lmodern,graphics,amsfonts,enumerate,verbatim,calc
} 
\usepackage[all]{xy}
\usepackage[colorlinks=true, citecolor=PineGreen, filecolor=PineGreen, linkcolor=Purple,pagebackref, hyperindex]{hyperref}
\usepackage[dvipsnames]{xcolor}
\usepackage{todonotes}
\usepackage{tikz-cd}
\usepackage{color, hyperref}
\usepackage[OT2,OT1]{fontenc}

\newcommand\cyr{%
\renewcommand\rmdefault{wncyr}%
\renewcommand\sfdefault{wncyss}%`
\renewcommand\encodingdefault{OT2}%
\normalfont
\selectfont}
\DeclareTextFontCommand{\textcyr}{\cyr}
\usepackage{amssymb,amsmath}
\DeclareFontFamily{OT1}{rsfs}{}	
\DeclareFontShape{OT1}{rsfs}{n}{it}{<-> rsfs10}{}
\DeclareMathAlphabet{\fmathscr}{OT1}{rsfs}{n}{it}
\numberwithin{equation}{section}
\newtheorem{Theoremx}{Theorem}
 
\newtheorem{theorem}{Theorem}[section]
\newtheorem{lemma}[theorem]{Lemma}
\newtheorem{notation}[theorem]{Notation}
\newtheorem{proposition}[theorem]{Proposition}
\newtheorem{corollary}[theorem]{Corollary}

\theoremstyle{definition}
\newtheorem{definition}[theorem]{Definition}
\newtheorem{remark}[theorem]{Remark}
\theoremstyle{remark}

\newtheorem{example}[theorem]{Example}

\newcommand{\Spec}{\operatorname{Spec}}

\newcommand{\Ht}{\operatorname{ht}}
\newcommand{\Height}{\operatorname{ht}}

\newcommand{\HSL}{\operatorname{HSL}}
\newcommand{\Fte}{\operatorname{Fte}}

\newcommand{\Ext}{\operatorname{Ext}}

\newcommand{\Tor}{\operatorname{Tor}}

\newcommand{\Ann}{\operatorname{Ann}}

\newcommand{\depth}{\operatorname{depth}}

\newcommand{\N}{\mathbb{N}}

\newcommand{\fm}{\mathfrak{m}}
\newcommand{\fp}{\mathfrak{p}}
\newcommand{\fq}{\mathfrak{q}}

\newcommand{\fn}{\mathfrak{n}}

\begin{document}
\title{$F$-nilpotent rings and permanence properties}

\author[Jennifer Kenkel]{Jennifer Kenkel}
%\thanks{}
\address{Department of Mathematics, University of Kentucky, Lexington, KY 40508, USA}
\email{j.kenkel@uky.edu}

\author[Kyle Maddox]{Kyle Maddox}
%\thanks{}
\address{Department of Mathematics, University of Missouri, Columbia, MO 65211, USA}
\email{kylemaddox@mail.missouri.edu}

\author[Thomas Polstra]{Thomas Polstra}
\thanks{Polstra was supported by NSF Postdoctoral Research Fellowship DMS $\#1703856$.}
\address{Department of Mathematics, University of Utah, Salt Lake City, UT 84102 USA}
\email{polstra@math.utah.edu}

\author[Austyn Simpson]{Austyn Simpson}
\thanks{Simpson was supported by NSF RTG grant DMS $\# 1246844$.}
\address{Department of Mathematics, University of Illinois at Chicago, Chicago, IL 60607 USA}
\email{awsimps2@uic.edu}

%\thanks{2010 {\em Mathematics Subject Classification\/}:13A35, 13D45.}
%\keywords{}

\begin{abstract} We explore the singularity classes $F$-nilpotent, weakly $F$-nilpotent, and generalized weakly $F$-nilpotent under faithfully flat local ring maps. As an application, we show that the loci of primes in a Noetherian ring of prime characteristic which define either weakly $F$-nilpotent or $F$-nilpotent local rings are open with respect to the Zariski topology whenever $R$ is $F$-finite or essentially of finite type over an excellent local ring.
\end{abstract}

 \maketitle
%\onehalfspacing

%{\color{black} \tableofcontents}

\section{Introduction}
Let $(R,\fm,k)$ denote a commutative Noetherian local ring of prime characteristic $p>0$. For each $e\in \mathbb{N}$ let $F^e:R\to R$ denote the $e$th iterate of the Frobenius endomorphism of $R$. Iterates of the Frobenius endomorphism induce Frobenius linear maps of local cohomology 
\[
F^e:H^i_\fm(R)\to H^i_\fm(R)
\]
and behavior of these Frobenius actions are used to classify various prime characteristic singularity classes. Among these so-called \emph{$F$-singularities} include $F$-rational, $F$-injective, and $F$-nilpotent singularities, the last of which being the primary subject of study in this article.

A ring $R$ is weakly $F$-nilpotent if high enough iterates of Frobenius actions on the lower local cohomology modules with support in the maximal ideal are eventually $0$. Utilizing the language of \cite{Lyubeznik}, a local ring $(R,\fm,k)$ of dimension $d$ is weakly $F$-nilpotent if and only if its $F$-depth is equal to $d$.  Lyubeznik developed the notion of $F$-depth to answer a question of Grothendieck. In particular, it was shown that if $(S,\fm,k)$ is a regular local ring of Krull dimension $d$, prime characteristic $p>0$, and $I\subseteq S$ an ideal then $H^i_I(S)=0$ for all $0\leq i < d$ if and only if $R=S/I$ is weakly $F$-nilpotent, \cite[Theorem~1.1]{Lyubeznik}.

A local ring $(R,\fm,k)$ of dimension $d$ is called $F$-nilpotent if $R$ is weakly $F$-nilpotent and the nilpotent submodule of $H^d_\fm(R)$ agrees with the tight closure of the $0$-submodule. Srinivas and Takagi conjectured (and proved for dimension smaller than four, \cite[Proposition 3.8]{SrinivasTakagi})  that a normal isolated singularity $X$ defined over $\mathbb{C}$ has $F$-nilpotent type if and only if given a log resolution $\pi:Y\to X$ one has $H^i(E,\mathcal{O}_E)=0$ for all $i\geq 1$, where $E$ is the exceptional divisor of $\pi$.

%More background before introducing results. Mention F-nilpotent + F-injective = F-rational. Mention Quy's result on multiplicity bounds for F-nilpotent rings. I'll write this.

The class of $F$-nilpotent singularities also deserves interest due to its relation to other classes of singularities. For instance, a reduced (alternatively, excellent and equidimensional) Noetherian local ring is $F$-rational if and only if it is $F$-injective and $F$-nilpotent. Further, $F$-rationality and $F$-injectivity are known to behave nicely in many aspects such as ascent and descent for sufficiently nice faithfully flat maps as well as openness of their respective loci (see e.g. \cite{Velez} and \cite{DattaMurayama}). Thus, it is natural to wonder if $F$-nilpotence is similarly well-behaved. Exploring this topic is the main content of this paper, and we collect a number of results along the way concerning the more general classes of \emph{weakly $F$-nilpotent} and \emph{generalized weakly $F$-nilpotent} rings.

Our first result is that the $F$-nilpotent locus of prime ideals is open with respect to the Zariski topology.

\begin{Theoremx}
\label{Open loci results}
Let $R$ be a ring of prime characteristic $p>0$ which is either $F$-finite or essentially of finite type over an excellent local ring. Then the following subsets of $\Spec(R)$ are open with respect to the Zariski topology:
\begin{enumerate}
\item $\{\fp\in \Spec(R)\mid R_\fp\mbox{ is weakly $F$-nilpotent}\};$
\item $\{\fp\in \Spec(R)\mid R_\fp\mbox{ is $F$-nilpotent}\}.$
\end{enumerate}
\end{Theoremx}

\begin{comment}
\begin{Theoremx}
\label{Main Theorem Descent}
Let $(R,\fm)\to (S,\fn)$ be a faithfully flat map of local rings of prime characteristic $p>0$ with Cohen-Macaulay closed fiber. 
Then
\begin{enumerate}
\item if $S$ is $F$-nilpotent then $R$ is $F$-nilpotent;
\item if $S$ is weakly $F$-nilpotent then $R$ is weakly $F$-nilpotent;
\item if $S$ is generalized weakly $F$-nilpotent then $R$ is generalized weakly $F$-nilpotent.
\end{enumerate}
\end{Theoremx}
\end{comment}

The proof of Theorem~\ref{Open loci results} in the $F$-finite scenario utilizes that Frobenius actions on local cohomology are realized as the Matlis dual of Cartier linear maps of the dualizing complex (see \cite[Lemma 2.3]{SrinivasTakagi}). Outside of the $F$-finite scenario we must understand the behavior of Frobenius actions on local cohomology under faithfully flat extensions. To this end, we study ascent and descent properties of local rings which are either $F$-nilpotent, weakly $F$-nilpotent, or generalized weakly $F$-nilpotent. Our results in this direction are as follows:

\begin{Theoremx}
\label{Main Theorem Ascent}
Let $(R,\fm)\to (S,\fn)$ be a faithfully flat map of local rings of prime characteristic $p>0$. 

\begin{enumerate}

\item If $S$ is weakly $F$-nilpotent then $R$ is weakly $F$-nilpotent. 

\item If $S$ is $F$-nilpotent, then $R$ is $F$-nilpotent under either of the following additional assumptions:
\begin{itemize}
    \item $S$ is an excellent local ring or;
    \item The closed fiber $S/\fm S$ is Cohen-Macaulay.
\end{itemize}
\item If $R$ is weakly $F$-nilpotent and $S/\fm S$ is Cohen-Macaulay then $S$ is weakly $F$-nilpotent.
\item If $R$ is $F$-nilpotent, $R/\sqrt{0}$ and $S\otimes_R R/\sqrt{0}$ have a test element in common, and $R\to S$ has geometrically regular fibers, then $S$ is $F$-nilpotent.
\end{enumerate}
\end{Theoremx}

We prove additionally that the property of being (weakly) $F$-nilpotent ascends under faithfully flat purely inseparable local ring homomorphisms; see Theorem~\ref{faithfully flat purely inseparable} for a precise statement. Moreover, we study the behavior of generalized weakly $F$-nilpotent rings under faithfully flat maps; see Theorem~\ref{theorem generalized weakly F-nilpotent faithfully flat map} for a precise statement.

The organization of the paper is as follows: In Section~\ref{Section preliminary} we discuss relevant background material needed throughout the article. In Section~\ref{Section Dualizing complexes} we utilize Cartier linear maps on dualizing complexes to begin our study of Frobenius actions on local cohomology. Of particular interest in this section is the existence of uniform bounds on the Frobenius test exponents of the localizations of a ring at prime ideals in the weakly $F$-nilpotent locus. The proofs of the ascent and descent statements in Theorem~\ref{Main Theorem Ascent} are gathered from the results of Section~\ref{Section flat base change results}, up to a Cohen-Macaulay assumption on the closed fiber. We also use Section~\ref{Section flat base change results} to further develop the theory of generalized weakly $F$-nilpotent rings as introduced by the second named author in \cite{Maddox}. The final section of the article, Section~\ref{Section open results}, is where we record a proof of Theorem~\ref{Open loci results}. Finally, we use our open loci results of Section~\ref{Section open results} to complete the proof of Theorem~\ref{Main Theorem Ascent}.

\section{Preliminary results}\label{Section preliminary}
In this section we recall basic facts about local cohomology and tight closure which will be essential in defining the singularity classes of interest, as well as in the proofs of our main theorems. We also give an overview of the recent literature concerning these singularities.

\subsection{Local cohomology and the Nagel-Schenzel isomorphism, and Frobenius actions}

Let $R$ be a commutative Noetherian ring, $I\subseteq R$ an ideal, and $M$ an $R$-module. We denote by $H^i_I(M)$ the $i$th local cohomology module of $M$ with support in $I$. If $(R,\fm,k)$ is local of Krull dimension $d$, $x_1,\ldots,x_d$ a system parameters, and $x=x_1\cdots x_d$ then the top local cohomology module is explicitly identified as a direct limit system:
\[
H^d_\fm(R)\cong \varinjlim \left(\frac{R}{(x_1^t,\ldots,x_d^t)}\xrightarrow{\cdot x}\frac{R}{(x_1^{t+1},\ldots,x_d^{t+1})}\right)
\]
Filter regular sequences allow us to identify lower local cohomologies with support in the maximal ideal as an explicit direct limit similar to the above.  Suppose $M$ is a finitely generated $R$-module. An element $x$ is a \emph{filter regular element} of $M$ if $(0:_Mx)$ has finite length. A sequence of elements $x_1,\ldots,x_i$ is a \emph{filter regular sequence} of $M$ if $x_j$ is a filter regular element of $M/(x_1,\ldots,x_{j-1})M$ for each $1\leq j\leq i$. A submodule $(0:_Mx)$ has finite length if and only if $x$ avoids all non-maximal associated primes of $M$. Therefore a simple prime avoidance argument can be used to show that $R$ admits a filter regular sequence on $M$ of length $\dim(M)$. The Nagel-Schenzel isomorphism stated below allows us to identify lower local cohomology modules with support in the maximal ideal as submodules of the top local cohomology on an ideal generated by a filter regular sequence.

\begin{theorem}
[{\cite[Proposition~3.4]{NagelSchenzel}}]
\label{Nagel-Schenzel} Let $(R,\fm,k)$ be a local ring of Krull dimension $d$. Suppose $x_1,\ldots,x_i$ is a filter regular sequence of length $i$ and $x=x_1\cdots x_i$, then
\[
H^i_\fm(R)\cong H^0_\fm(H^i_{(x_1,\ldots,x_i)}(R))\cong \varinjlim\left(\frac{(x_1^t,\ldots,x_i^t):_R\fm^\infty}{(x_1^{t},\ldots,x_i^{t})}\xrightarrow{\cdot x}\frac{(x_1^{t+1},\ldots,x_i^{t+1}):_R\fm^\infty}{(x_1^{t+1},\ldots,x_i^{t+1})}\right).
\]
\end{theorem}

We will repeatedly use the following lemma whose content is that a filter regular sequence is preserved under a faithfully flat extension of local rings with $0$-dimensional closed fiber. This fact in combination with Theorem~\ref{Nagel-Schenzel} will be useful when establishing the permanence properties of Section~\ref{Section flat base change results}.

\begin{lemma}
\label{Lemma preserve filter regular sequence}
Let $(R,\fm)\to (S,\fn)$ be a faithfully flat local ring map with $0$-dimensional closed fiber $S/\fm S$. If $x_1,\ldots, x_d$ is a filter regular sequence in $R$ then $x_1,\ldots, x_d$ is a filter regular sequence in $S$.
\end{lemma}
\begin{proof}
The sequence $x_1,\ldots,x_d$ is a filter regular sequence if and only if the colon ideals $((x_1,\ldots, x_i):_Rx_{i+1})/(x_1,\ldots, x_i)$ of $R/(x_1,\ldots, x_i)$ has finite length for each $1\leq i\leq d-1$. Base changing by a faithfully flat map preserves colon ideals. Because $S/\fm S$ is $0$-dimensional the images of $x_1,\ldots, x_d$ satisfy the same finite length criterion in $S$ and therefore $x_1,\ldots, x_d$ is a filter regular sequence of $S$.
\end{proof}

\subsection{Frobenius actions on local cohomology, Frobenius closure, and tight closure} 

Throughout this subsection suppose that $(R,\fm,k)$ is a $d$-dimensional local ring of prime characteristic $p>0$ and $F^e: R\to R$ is the $e$th iterate of the Frobenius endomorphism. Then $F^e$ induces $p^e$-linear maps $F^e: H^i_\fm(R)\to H^i_\fm(R)$, see \cite[Lecture~21]{24Hours}. Equivalently, the Frobenius endomorphism induces $R$-linear maps $H^i_\fm(R)\to F^e_*H^i_\fm(R)$ where $F^e_*H^i_\fm(R)$ is the Frobenius pushforward of $H^i_\fm(R)$. Suppose that $x_1,\ldots,x_i$ a filter regular sequence of $R$ and identify $H^i_\fm(R)$ as in Theorem~\ref{Nagel-Schenzel}. If $\gamma\in H^i_\fm(R)$ is represented by the class of $\eta+(x_1^t\ldots,x_i^t)$ in the above direct limit system then the $e$th Frobenius action of $R$ sends $\gamma$ to the element of $H^i_\fm(R)$ represented by $(\eta^{p^e}+(x_1^{tp^e},\ldots,x_i^{tp^e}))$. 

% We denote by $0^{F}_{H^i_\fm(R)}$ submodule of $H^i_\fm(R)$ consisting of elements which are mapped to $0$ under a large enough iterate of the Frobenius action on $H^i_\fm(R)$.

Frobenius actions on local cohomology modules are often used to classify prime characteristic singularities (see \cite{BlickleBondu, EnescuHochster,Maddox,PolstraQuy,SmithFrational}). This article concerns itself with the following singularity classes defined in terms of Frobenius actions on local cohomology:
\begin{enumerate}
\item $F$-nilpotent, \cite{BlickleBondu};
\item weakly $F$-nilpotent, \cite{PolstraQuy};
\item generalized weakly $F$-nilpotent, \cite{Maddox}.
\end{enumerate}

We will be interested in the behavior of the above singularity classes under flat base change in Section~\ref{Section flat base change results}. Understanding the algebraic technicalities of local cohomology modules under the identifications of Theorem~\ref{Nagel-Schenzel} will allow us to do just that. To define these classes of singularities, we first recall the basic notions of tight closure theory. See \cite{HochsterHunekeJAMS} for a thorough treatment.

%I commented these things out because tight closure and Frobenius closure had not been defined at this point in the exposition.
% If in addition to being weakly $F$-nilpotent, the Frobenius closure and tight closure of the $0$-submodule of $H^d_\fm(R)$ agree, then we say $R$ is \emph{$F$-nilpotent}. 

% For basics of tight closure we refer the reader to \cite{HochsterHunekeJAMS}. 

% {\color{blue} Apriori, $e$ depends on $\eta$. The amazing fact is that there is a uniform $e$ which works for all $\eta$ because HSL-numbers are finite. But this is a non-trivial fact that we do not mention until Section~3, so I would prefer to not mention it here.}

Let $R^\circ$ be the complement of the union of the minimal primes of $R$ and let $F^e(-)=-\otimes_R F^e_* R$ denote the $e$th Frobenius pullback functor, i.e. base change along the $e$th iterate of Frobenius. If $N\subseteq M$ are $R$-modules then we say an element $m\in M$ is in $N^*_M$, the tight closure of $N$, if there exists a $c\in R^\circ$ so that $m$ is in the kernel of the following composition of maps for all $e\gg 0$:
\[
M\to M/N \to F^e(M/N)\xrightarrow{\cdot c} F^e(M/N).
\]
Similarly, we say $m\in M$ is in $N^F_{M}$, the Frobenius closure of $N$, if $m$ is in the kernel of the following maps for all $e\gg 0$:
\[
M\to M/N \to F^e(M/N).
\]
Observe that there are inclusions $N\subseteq N^F_M\subseteq N^*_M$.

The top local cohomology module with support in the maximal ideal enjoys the following property: $F^e(H^d_\fm(R))\cong H^d_\fm(R)$ and $H^d_\fm(R)\to F^e(H^d_\fm(R))$ is identified with the $e$th Frobenius action of the top local cohomology module. In particular, if $\gamma \in H^d_\fm(R)$ then $\gamma\in 0^*_{H^d_\fm(R)}$ if and only if there exists a $c\in R^\circ$ so that $\gamma$ is in the kernel of the following composition of $R$-linear maps for all $e\gg 0$:
\[
H^d_\fm(R)\xrightarrow{F^e} F^e_*H^d_\fm(R)\xrightarrow{\cdot F^e_*c} F^e_*H^d_\fm(R).
\] 
Similarly, $\gamma\in 0^F_{H^d_\fm(R)}$ if and only if $\gamma$ is the kernel of the following maps for all $e\gg 0$:
\[
H^d_\fm(R)\xrightarrow{F^e} F^e_*H^d_\fm(R).
\]
% In particular, $0^F_{H^d_\fm(R)}=0^{F_R}_{H^d_\fm(R)}$, a property not typically enjoyed by lower local cohomologies.
%I commented this out because relative Frobenius closure isn't introduced in this section yet.

\begin{definition}
Let $(R,\fm,k)$ be a local ring of prime characteristic $p>0$ and dimension $d$. We say that $R$ is \emph{weakly $F$-nilpotent} if $0^{F}_{H^i_\fm(R)}=H^i_\fm(R)$ for all $0\leq i <d$. If $R$ is weakly $F$-nilpotent and $0^*_{H^d_\fm(R)}=0^F_{H^d_\fm(R)}$ then we say that $R$ is \emph{$F$-nilpotent}.
\end{definition}

An equivalent description of weakly $F$-nilpotent rings is as follows: suppose $x_1,\ldots,x_i$ is a filter regular sequence of length $i$, $x=x_1\cdots x_i$, and identify $H^i_\fm(R)$ as in Theorem~\ref{Nagel-Schenzel}. By definition, $R$ is weakly $F$-nilpotent if and only if for every $\eta\in (x_1,\ldots,x_i):_R\fm^\infty$ and $e\gg 0$ there exists $t\in\N$ so that $\eta^{p^e}x^{tp^e}\in (x_1^{(t+1)p^e},\ldots,x_i^{(t+1)p^e})$. Moreover, one observes by chasing through the direct limit identifications of $H^d_\fm(R)$ that $R$ is $F$-nilpotent if and only if $R$ is weakly $F$-nilpotent and satisfies the following property: for every system of parameters $x_1,\dots, x_d$, every $\gamma\in R$, $c\in R^\circ$, then
\begin{align*}
    \left(\text{for every } e\gg 0\text{ there exists } s\in\N \text{ s.t } c\gamma^{p^e}x^{sp^e}\in(x_1^{(s+t)p^e},\dots,x_d^{(s+t)p^e})\right)\Rightarrow\\
    \left(\text{for every } e\gg 0\text{ there exists } s\in\N \text{ s.t } \gamma^{p^e}x^{sp^e}\in(x_1^{(s+t)p^e},\dots,x_d^{(s+t)p^e})\right).
\end{align*}

There is one further class of singularities of interest in this article. It is proven in \cite{Quy} that the Frobenius test exponent is finite for weakly $F$-nilpotent rings. The second named author refined this result by isolating the salient features of weakly $F$-nilpotent rings used in Quy's proof. Indeed, it is shown in \cite[Theorem~3.1 and Theorem~3.6]{Maddox} that the following class of singularities has finite Frobenius test exponent.

\begin{definition}
A local $d$-dimensional ring $(R,\fm,k)$ of prime characteristic $p>0$ is called \emph{generalized weakly $F$-nilpotent} if for each $0\leq i < d$ the $R$-module $H^i_\fm(R)/0^{F}_{H^i_\fm(R)}$ has finite length.
\end{definition}
In this article we characterize generalized weakly $F$-nilpotent rings as local rings which are weakly $F$-nilpotent on the punctured spectrum in Proposition~\ref{proposition gwfn = wfn on punctured spectrum}.  

% Generalized weakly $F$-nilpotent rings were introduced by the second named author in \cite{Maddox} for the purpose of finding new classes of rings which have finite Frobenius Test exponents, c.f \cite{HunekeKatzmanSharpYao, KatzmanSharp, Quy}.

\subsection{Relative Frobenius Actions}
\label{Subsection relative Frobenius}

Let $R\to S$ be a homomorphism of rings of prime characteristic $p>0$. We denote by $F^e_{S/R}$ the map which fills in the following cocartesian diagram:
\begin{center}
\begin{tikzcd}
R \arrow[r, "F^e"] \arrow[d]
& F^e_*R \arrow[d] \arrow[ddrr, bend left] \\
S \arrow[r] \arrow[drrr, bend right, "F^e"']
& F^e_*R\otimes_R S \arrow[drr, dashed,"{F^e_{S/R}}" description]\\ 
&&& F^e_*S
\end{tikzcd}
\end{center}

The Radu-Andr\'e Theorem equates faithful flatness of $F^e_{S/R}$ with the property that $R\to S$ is faithfully flat with geometrically regular fibers, a theorem we record for future reference.

\begin{theorem}
[{\cite[Theorem~4]{RaduRelative}, \cite[Theorem~1]{AndreRelative}}]
\label{Theorem Radu-Andre}
Let $R\to S$ be a homomorphism of rings of prime characteristic $p>0$. Then $R\to S$ is faithfully flat with geometrically regular fibers if and only if the relative Frobenius map $F^e_{S/R}$ is faithfully flat for all $e\in \mathbb{N}$.
\end{theorem}

We will often employ another notion of relative Frobenius action when using the Nagel-Schenzel isomorphism. Let $I\subseteq R$ be an ideal of a local ring $(R,\fm,k)$ of prime characteristic $p>0$. Then for each $e\in\N$ the Frobenius map $F^e: R/I\to R/I$ can be factored through a Frobenius linear map relative to $R$:
\[
R/I\xrightarrow{F^e_R}R/I^{[p^e]}\to R/I.
\]

The first map is given by $F^e_R(a+I)=a^{p^e}+I^{[p^e]}$ and the second map is the natural projection. Note that we are suppressing the reference to the ideal $I$ in this notation. Relative Frobenius maps induce Frobenius linear maps of local cohomology modules:
\[
H^i_\fm(R/I)\xrightarrow{F^e_R}F^e_*H^i_\fm(R/I^{[p^e]}).
\]

We now discuss the notion of relative tight closure of a local cohomology module as defined in \cite{PolstraQuy}. Suppose $R/I$ is $d$-dimensional. An element $\gamma\in H^d_\fm(R/I)$ is an element of $0^{*_R}_{H^d_\fm(R/I)}$, \emph{the tight closure of $0$ relative to $R$}, if there exists $c\in R^\circ$ so that $\gamma$ is in the kernel of the composition
\[
H^d_\fm(R/I)\xrightarrow{F^e_R}H^d_\fm(R/I^{[p^e]})\xrightarrow{\cdot F^e_*c}H^d_\fm(R/I^{[p^e]})
\]
for all $e\gg 0$. Similarly, $\gamma\in H^d_\fm(R/I)$ is an element of $0^{F_R}_{H^d_\fm(R/I)}$, \emph{the Frobenius closure of $0$ relative to $R$}, if $\gamma$ is in the kernel of the following maps for all $e\gg0$:
\[
H^d_\fm(R/I)\xrightarrow{\cdot F^{e}_R}H^d_\fm(R/I^{[p^e]}).
\]
We say $R/I$ is \emph{weakly $F$-nilpotent relative to $R$}\footnote{We use this terminology in order to be consistent with \cite{PolstraQuy}.} if for each $0\leq i<d$ and $\gamma\in H^i_\fm(R/I)$ there exists $e\gg 0$ so that $F^e_R(\gamma)$ is the $0$-element of $H^i_\fm(R/I^{[p^e]})$. We say $R/I$ is \emph{$F$-nilpotent relative to $R$} if $R/I$ is weakly $F$-nilpotent relative to $R$ and $0^{F_R}_{H^d_{\fm}(R/I)}=0^{*_R}_{H^d_{\fm}(R/I)}$. The third named author and Pham Hung Quy utilized relative Frobenius actions and filter regular sequences to provide the following characterizations of weakly $F$-nilpotent and $F$-nilpotent rings:

%We remark that $0^{F}_{H^i_\fm(R)}\subseteq 0^{F_R}_{H^i_\fm(R)}$. The inclusion could be strict if $0\leq i< d$ and equality is obtained when $i=d$. 

\begin{theorem}[{\cite[Theorem~5.9]{PolstraQuy}}]
\label{Theorem criterion of F-nilpotence}
Let $(R,\fm,k)$ be an excellent equidimensional local ring of prime characteristic $p>0$ and dimension $d$. Let $x_1,\ldots,x_i$ be a filter regular sequence of length $i$. Then 
\begin{enumerate}
\item $R$ is weakly $F$-nilpotent if and only if $R/(x^N_1,\ldots,x^N_i)$ is weakly $F$-nilpotent relative to $R$ for all $N\in \mathbb{N}$;
\item $R$ is $F$-nilpotent if and only if $R/(x^N_1,\ldots,x^N_i)$ is $F$-nilpotent relative to $R$ for all $N\in \mathbb{N}$.
\end{enumerate}
\end{theorem}

\section{Dualizing complexes and Frobenius actions}\label{Section Dualizing complexes}
Let $R$ be a ring of prime characteristic $p>0$. A \emph{Cartier linear map} $\varphi:M\to M$ on an $R$-module $M$ is an $R$-linear map $F^e_*M\to M$. Suppose further that $(R,\fm,k)$ is a complete local ring of prime characteristic $p>0$. We discuss how Frobenius actions on the local cohomology modules of $R$ can be understood through Cartier linear maps of $\Ext$-modules induced by Matlis duality (see for example \cite[Section~5]{BlickleBockle} and \cite[Section~2]{GlobalParameterTestIdeals}).  

Any (not necessarily local) $F$-finite ring $R$ is the homomorphic image of a regular ring \cite{Gabber}, hence $R$ admits a dualizing complex\footnote{ We refer the reader to \cite{HartshorneResiduesAndDuality} and \cite[\href{https://stacks.math.columbia.edu/tag/0A7A}{Tag 0A7A}]{StacksProject} for basics of dualizing complexes.} $\omega_R^\bullet$. Iterates of the Frobenius endomorphism induce Cartier linear maps on the chain complex $\omega_R^\bullet$ and such maps can be localized to understand the behavior of Frobenius actions on local cohomology modules. We denote by $(F^e)^\vee:F^e_*\omega_R^\bullet\to \omega_R^\bullet$ the induced Cartier linear map obtained by evaluating at $1$. The Cartier linear map $(F^e)^\vee$ is a degree preserving map on the chain complex $\omega^\bullet_R$. In particular, for each $i\in \mathbb{Z}$ we let $(F^e)^\vee(i)$ denote degree $i$ piece of $(F^e)^\vee$. As our notation is suggestively indicating, in a local ring the Matlis dual of $(F^e)^\vee$ will be the Frobenius action on local cohomology.

If $(R,\fm,k)$ is a local ring of Krull dimension $d$ with dualizing complex $\omega_R^\bullet$ then we say $\omega_R^\bullet$ is normalized if $H^{-d}(\omega_R^\bullet)=\omega_R$ is a canonical module for $R$. If $R$ is not assumed to be local, but is assumed to be locally equidimensional of Krull dimension $d$, then we say a dualizing complex $\omega_R^\bullet$ is normalized if for every $\fp\in\Spec(R)$ the shifted localized complex $(\omega_R^{\bullet}\otimes R_\fp)[-d+\Height(\fp)]$ is a normalized complex for $R_\fp$.

\begin{proposition}
\label{Proposition Frobenius is just Grothendieck Trace} Let $R$ be a locally equidimensional $F$-finite ring of Krull dimension $d$, $\omega_R^\bullet$ a normalized dualizing complex of $R$, and for each $e\in \mathbb{N}$ let $(F^e)^\vee:F^e_*\omega_R^\bullet \to \omega_R^\bullet$ be the map of dualizing complexes induced by $F^e:R\to R$. Then for each $\fp\in \Spec(R)$ and $i\in \mathbb{N}$ the $R_\fp$-Matlis dual of $(F^e)^\vee_\fp(-i)$ of the normalized dualizing complex $\omega_{R_\fp}^\bullet=(\omega_{R}^\bullet\otimes R_\fp)[-d+\Height(\fp)]$ induces the $e$th Frobenius action on the local cohomology module $F^e:H^{i}_{\fp R_\fp}(R_\fp)\to F^e_*H^{i}_{\fp R_\fp}(R_\fp)$.
\end{proposition}

\begin{proof}
Dualizing complexes, the evaluation at $1$ map $(F^e)^\vee:F^e_*\omega_R^\bullet\to \omega_R^\bullet$, and the Frobenius map are functorial and commute with localization and completion. Therefore the proposition is a consequence of \cite[Lemma~5.1]{BlickleBockle}.
\end{proof}

As a consequence of Proposition~\ref{Proposition Frobenius is just Grothendieck Trace} we demonstrate the existence of uniform bounds of the Hartshorne-Speiser-Lyubeznik numbers in a (not necessarily local) ring $R$ which is either $F$-finite or essentially of finite type over an excellent local ring.

\begin{definition}\cite{HartshorneSpeiser, LyubeznikFmodules}
Let $(R,\fm,k)$ be a $d$-dimensional local ring of prime characteristic $p>0$. For each $0\leq i\leq d$, consider the non-decreasing sequence of $R$-submodules
$$N_{i,e}:=\{z\in H^i_\fm(R)\mid F^e(z)=0\}\subseteq H^i_\fm(R).$$
\begin{enumerate}
    \item The $i$th Hartshorne-Speiser-Lyubeznik of $R$ is defined to be $$\HSL_i(R):=\min\{e\mid N_{i,e+j}=N_{i,e}\text{ for all }j\geq 1\};$$
    \item The Hartshorne-Speiser-Lyubeznik of $R$ is defined as $$\HSL(R):=\max\{\HSL_i(R)\mid 0\leq i\leq d\}.$$
\end{enumerate}
\end{definition}

\begin{corollary}
\label{Corollary finite HSL number}
Let $R$ be a locally equidimensional ring of prime characteristic $p>0$ and Krull dimension $d$ which is either $F$-finite or essentially of finite type over an excellent local ring. Then the set of numbers $\{\HSL(R_\fp)\mid \fp\in \Spec(R)\}$ is bounded. 
\end{corollary}

\begin{proof}
Suppose first that $R$ is $F$-finite. Let $\omega_R^\bullet$ be a normalized dualizing complex of $R$. Utilizing the language of \cite{BlickleBockle}, the graded pieces of the complex $\omega_R^\bullet$ along with the corresponding graded pieces of $(F^e)^\vee$ are coherent Cartier modules on $X=\Spec(R)$. Therefore we may apply \cite[Proposition~2.14]{BlickleBockle} in degrees $-d\leq i \leq 0$ to know there exists an integer $e_0$ so that the image of $(F^e)^\vee:F^e_*\omega_R^\bullet \to \omega_R^\bullet$ agrees with the image of $(F^{e_0})^\vee$ in degrees $-d\leq i\leq 0$ for all $e\geq e_0$. By Proposition~\ref{Proposition Frobenius is just Grothendieck Trace} measuring the $\HSL$ numbers of $R_\fp$ is equivalent to understanding when the cokernel of $(F^e)^\vee$ stabilizes, from which the corollary is easily derived.

Suppose $R$ is essentially of finite type over an excellent local ring $A$. The $\HSL$-numbers of a local ring are easily seen to be preserved under faithfully flat local ring maps with $0$-dimensional fiber, see \cite[Discussion following Remark~2.1]{KatzmanZhang}. We may base change by the completion of $A$ without affecting the hypotheses and assume $R$ is essentially of finite type over a complete local ring. We can then utilize a $\Gamma$-construction to find a faithfully flat and purely inseparable map $R\to R^\Gamma$ so that $R^\Gamma$ is $F$-finite\footnote{For more on the existence of such ring maps, see either \cite{HochsterHunekeTransactions} or \cite{TakumiGamma}.}.
\end{proof}

\begin{definition}
Let $I\subseteq R$ be an ideal and $c\in R^\circ$. We say that $e_0$ is \emph{test exponent for tight closure with respect to $c$ for $I$} if $I^*=\{x\in R\mid cx^{p^{e_0}}\in I^{[p^{e_0}]}\}$.
\end{definition}
Sharp proved in \cite{SharpTightClosureTestExponents} that test exponents of parameters ideals in a local ring with respect to a test element depend on the $\HSL$-numbers of that ring. Sharp's result and Corollary~\ref{Corollary finite HSL number} provide a uniform test exponent for tight closure of parameter ideals among all localizations of an $F$-finite ring $R$.

\begin{corollary}
\label{Finite test exponents for tight closure}
Let $R$ be an $F$-finite locally equidimensional ring and suppose that $R$ has a completely stable test element $c$ (e.g., if $R$ is reduced). There exists an $e_0\in \mathbb{N}$ with the following properties:
\begin{enumerate}
\item If $\fp\in \Spec(R)$ and $\fq=(x_1,\ldots,x_{\Height(\fp)})R_\fp$ is a parameter ideal of $R_\fp$ then $e_0$ is a test exponent for tight closure with respect to $c$ for $\fq^*$;
\item If $\fp \in \Spec(R)$ then $e_0+1$ is a test exponent for tight closure with respect to $c$ for $0^*_{H^d_{\fp R_\fp}(R_\fp)}$.
\end{enumerate}
\end{corollary}

\begin{proof}
Property (1) is immediate by Corollary~\ref{Corollary finite HSL number} and \cite[Corollary~2.4(i)]{SharpTightClosureTestExponents}. We will show that the proof of (2) reduces to showing that if $(R,\fm,k)$ is local of dimension $d$, $e_0$ a test exponent for parameter ideals with respect to $c$, then $e_0+1$ is a test exponent for $0^*_{H^d_{\fm}(R)}$.

Let $x_1,\cdots,x_d$ be a system of parameters, $\mathfrak{q} = (x_1,\cdots, x_d)$, and $x=x_1\cdot x_2\cdots x_d$. Suppose $\xi = [z+(x_1^n,\cdots,x_d^n)]\in H^d_{\fm}(R)$ and
\[
c\xi^{p^{e_0+1}}= [cz^{p^{e_0+1}}+(x_1^{np^{e_0+1}},\cdots,x_d^{np^{e_0+1}})]=0.
 \] 
We aim to show $\xi\in 0^*_{H^d_{\fm}(R)}$. There exists an integer $N$ such that 
\[
cz^{p^{e_0+1}}x^N \in (x_1^{np^{e_0+1}+N},\cdots,x_d^{np^{e_0+1}+N}).
\]
Thus, by the colon-capturing property of tight closure, \[cz^{p^{e_0+1}} \in (x_1^{np^{e_0+1}+N},\cdots,x_d^{np^{e_0+1}+N}):x^N \subset (x_1^{np^{e_0+1}},\cdots,x_d^{np^{e_0+1}})^*.\] 
Thus by the proof of \cite[Corollary~2.4(ii)]{SharpTightClosureTestExponents}, $z \in (x_1^n,\cdots,x_d^n)^*$, and so $[z+(x_1^n,\cdots,x_d^n)] \in 0^*_{H^d_\mathfrak{m}(R)}$.
\end{proof}

\begin{comment}
\begin{theorem}
\label{Theorem uniform bounds on HSL numbers}
Let $R$ be a locally equidimensional ring which is either $F$-finite or essentially of finite type over an excellent local ring. Then there exists an $N\in \mathbb{N}$ so that for each $\fp \in \Spec(R)$ the Hartshorne-Speizer-Lyubeznik numbers of $R_\fp$ is bounded by $N$.
\end{theorem}

\begin{proof}
The proposition statement in the $F$-finite scenario is the content of Corollary~\ref{Corollary finite HSL number}. Moreover, the $\HSL$ numbers of a local ring are easily seen to be preserved under faithfully flat local ring maps with $0$-dimensional fiber, see \cite[Discussion following Remark~2.1]{KatzmanZhang}. Hence if $R$ is essentially of finite type over an excellent local ring $A$ we may base change by the completion of $A$ and assume $R$ is essentially of finite type over a complete local ring. We can then utilize a $\Gamma$-construction to find a faithfully flat and purely inseparable map $R\to R^\Gamma$ so that $R^\Gamma$ is $F$-finite\footnote{For more on the existence of such ring maps, also referred to as $\Gamma$-constructions, see either \cite{HochsterHunekeTransactions} or \cite{TakumiGamma}.}.
\end{proof}
\end{comment}

Similar to the notion of test exponents for tight closure, there is a notion of a test exponent for Frobenius closure. Let $I$ be an ideal in a Noetherian ring $R$ of prime characteristic $p>0$. Recall that the Frobenius closure of $I$ is the ideal $I^F=\{r\in R\mid r^{p^e}\in I^{[p^e]}\mbox{ for some }e\in\mathbb{N}\}$. It is not difficult to see that there exists an $e_0$, depending on $I$, so that $I^{[p^{e_0}]}=(I^F)^{[p^{e_0}]}$.

\begin{definition}
Let $I\subseteq R$ be an ideal. The minimal $e_0\in \mathbb{N}$ so that $I^{[p^{e_0}]}=(I^F)^{[p^{e_0}]}$ is called the \emph{Frobenius test exponent of $I$} and is denoted $\Fte(I)$. If $R$ is local then we say that $e_0\in \mathbb{N}$ is a \emph{Frobenius test exponent of $R$} if $e_0$ bounds the Frobenius test exponent of every ideal of $R$ generated by a full system of parameters. If such a bound exists we denote by $\Fte(R)$ the minimum Frobenius test exponent of $R$.
\end{definition}  

It has been shown that the classes of Cohen-Macaulay, generalized Cohen-Macaulay, weakly $F$-nilpotent, and generalized weakly $F$-nilpotent rings all have finite Frobenius test exponent (see \cite{KatzmanSharp}, \cite{HunekeKatzmanSharpYao}, \cite{Quy}, and \cite{Maddox} respectively\footnote{One cannot expect a uniform Frobenius test exponent for the class of all ideals of $R$, see \cite{Brenner}.}). In \cite{Quy}, Pham Hung Quy obtains bounds for the Frobenius test exponent of a local weakly $F$-nilpotent ring in terms of its $\HSL$-numbers. Therefore we are able to uniformly bound the Frobenius test exponents of a ring $R$ at its localizations at prime ideals defining weakly $F$-nilpotent local rings.

\begin{corollary}
\label{corollary bounded FTE}
Let $R$ be a Noetherian ring of prime characteristic $p>0$ which is either $F$-finite or essentially of finite type over an excellent local ring. Then the set $$\{\Fte(R_\fp)\mid \fp\in\Spec (R)\text{ such that } R_\fp\text{ is weakly } F\text{-nilpotent}\}$$ is bounded.
\end{corollary}

\begin{proof}
By Corollary~\ref{Corollary finite HSL number} let $\HSL(R_\fq)<C$ for every $\fq\in \Spec R$. Let $\fp\in\Spec R$ and suppose that $R_\fp$ is weakly $F$-nilpotent. By the proof of \cite[Main Theorem]{Quy}, one obtains
\begin{equation}
\Fte(R_\fp)\leq \sum\limits_{i=0}^{\Ht\fp}{\Ht\fp\choose i}\HSL_i(R_\fp)\leq\sum\limits_{i=0}^{\Ht\fp}{\Ht\fp\choose i} \HSL(R_\fp) <C\cdot\sum\limits_{i=0}^{\Ht\fp}{\Ht\fp\choose i}.\label{Quy-Fte-bound}
\end{equation}
Every $F$-finite ring has finite Krull dimension by \cite[Proposition~1.1]{Kunz1976}, so suppose $R$ has Krull dimension $d$. Then the right hand side of \ref{Quy-Fte-bound} has at most $d+1$ terms which are all bounded by a constant independent of $\fp$, hence so is $\Fte(R_\fp)$.
\end{proof}

Pham Hung Quy has alerted us that Corollary~\ref{corollary bounded FTE} may also be recovered from the literature for a \emph{local} weakly $F$-nilpotent ring. See \cite[Main Theorem]{Quy} and \cite[Proposition~3.5]{HuongQuy}.

\section{Permanence properties}\label{Section flat base change results}

Theorem~\ref{Main Theorem Ascent} is a consequence of the material in this section. We begin with a lemma.

\begin{lemma}
\label{lemma faithfully flat CM fiber} Let $(R,\fm)\to (S,\fn)$ be a faithfully flat local ring map with Cohen-Macaulay closed fiber $S/\fm S$ of Krull dimension $\ell$. Then $\depth(S)= \depth(R) + \ell$ and there exists $T_1,\ldots, T_\ell$ a regular sequence of elements of $S$ which is also a regular sequence of $S/\fm S$. Moreover, for any such sequence of elements
\begin{enumerate}
\item $R\to S/(T_1,\ldots, T_\ell)$ is faithfully flat with $0$-dimensional closed fiber,
\item and if $x_1,\ldots,x_i$ is a parameter sequence of $R$ then $(x_1,\ldots,x_i,T_1,\ldots,T_\ell)S\cap R=(x_1,\ldots,x_i)R$. 
\end{enumerate}
\end{lemma}

\begin{proof}
Because $R\to S$ is faithfully flat we find that $\depth(S)=\depth(R)+\depth(S/\fm S)=\depth(R)+\ell$ by \cite[Theorem~23.3]{Matsumura}. By prime avoidance and the assumption that $S/\fm S$ is Cohen-Macaulay we may choose $T_1,\ldots, T_\ell$ in $S$ to be a regular sequence of $S$ and $S/\fm S$. Consider the short exact sequence
\[
0\to S\xrightarrow{\cdot T_1}S\to \frac{S}{(T_1)}\to 0.
\]
We are assuming $R\to S$ is flat. Equivalently, $\Tor_1^R(R/\fm,S)=0$ and therefore there's an exact sequence
\[
0\to \Tor_1^R\left(\frac{R}{\fm}, \frac{S}{(T_1)}\right)\to \frac{S}{\fm S}\xrightarrow{\cdot T_1}\frac{S}{\fm S}\to \frac{S}{(\fm,T_1)S}\to 0.
\]
But we have chosen $T_1$ to be $S/\fm S$-regular and therefore $\Tor_1^R(R/\fm, S/(T_1))=0$, i.e. $R\to S/(T_1)$ is faithfully flat. By induction on the dimension of the closed fiber we find that $R\to S/(T_1,\ldots,T_\ell)$ is faithfully flat with $0$-dimensional closed fiber.

Let $x_1,\ldots,x_i$ be a parameter sequence of $R$. Then modulo $(x_1,\ldots,x_i)$ the intersection ideal $(x_1,\ldots,x_i,T_1,\ldots,T_\ell)S\cap R$ is realized as the kernel of the faithfully flat map $R/(x_1,\ldots,x_i)\to S/(x_1,\ldots,x_i,T_1,\ldots, T_\ell)S$. Faithfully flat maps are injective and therefore $(x_1,\ldots,x_i,T_1,\ldots,T_\ell)S\cap R=(x_1,\ldots,x_i)R$.
\end{proof}

\begin{notation}
\label{Notation T's}
We denote a sequence of ring elements $T_1,\dots, T_\ell\in R$ by $\underline{T}$. Moreover, for any $N\in\N$ we let $\underline{T}^N=T_1^N,\dots, T_\ell^N$.
\end{notation}

We are now prepared to present a proof that the property of being either $F$-nilpotent or weakly $F$-nilpotent descends under faithfully flat local ring maps with Cohen-Macaulay closed fiber. The key ingredients are Theorem~\ref{Nagel-Schenzel}, Theorem~\ref{Theorem criterion of F-nilpotence}, and Lemma~\ref{lemma faithfully flat CM fiber}. Under mild hypotheses we will be able to remove the assumption that the closed fiber is Cohen-Macaulay in Section~\ref{Section open results} using our open loci results. 
\begin{theorem}
\label{theorem F-nilpotent descends}
Let $(R,\fm)\to (S,\fn)$ be a faithfully flat map of local rings of prime characteristic $p>0$ with Cohen-Macaulay closed fiber.
\begin{enumerate}
\item If $S$ is weakly $F$-nilpotent then $R$ is weakly $F$-nilpotent.
\item  If $S$ is $F$-nilpotent then $R$ is $F$-nilpotent.
\end{enumerate}  
\end{theorem}

\begin{proof}
Suppose $S/\fm S$ is an $\ell$-dimensional Cohen-Macaulay local ring. We utilize prime avoidance to choose $T_1,\ldots, T_\ell$ elements of $S$ which are regular on both $S$ and $S/\fm S$. Suppose further that $S$ is weakly $F$-nilpotent. Let $0\leq i< d$ and $x_1,\ldots, x_i$ a filter regular sequence of length $i$ of $R$ and let $x=x_1\cdots x_i$. By Theorem~\ref{Nagel-Schenzel}
\[
H^i_\fm(R)\cong H^0_\fm(H^i_{(x_1,\ldots,x_i)}(R))\cong  \varinjlim\left(\frac{(x_1^t,\ldots,x_i^t):_R\fm^\infty}{(x_1^t,\ldots,x_i^t)}\xrightarrow{\cdot x}\frac{(x_1^{t+1},\ldots,x_i^{t+1}):_R\fm^\infty}{(x_1^{t+1},\ldots,x_i^{t+1})}\right).
\]
Suppose that $\eta+(x_1,x_2,\ldots,x_i)$ represents an element of $H^i_\fm(R)$ in the above direct limit system. In particular, there exists an $N\in \mathbb{N}$ so that $\fm^N\eta\subseteq (x_1,\ldots, x_i)$. By (1) of Lemma~\ref{lemma faithfully flat CM fiber} the map $R\to S/(\underline{T})$ is faithfully flat with $0$-dimensional closed fiber and therefore $x_1,\ldots,x_i$ is a filter regular sequence of length $i$ in $S/(\underline{T})$ by Lemma~\ref{Lemma preserve filter regular sequence}. Invoking Theorem~\ref{Nagel-Schenzel} again we find that
\[
H^i_\fn(S/(\underline{T}))\cong  \varinjlim\left(\frac{(x_1^t,\ldots,x_i^t,\underline{T})S:_S\fn^\infty}{(x_1^t,\ldots,x_i^t,\underline{T})S}\xrightarrow{\cdot x}\frac{(x_1^{t+1},\ldots,x_i^{t+1},\underline{T})S:_S\fn^\infty}{(x_1^{t+1},\ldots,x_i^{t+1},\underline{T})S}\right).
\]
Because $R\to S/(\underline{T})$ has $0$-dimensional fiber we see that by increasing $N$ as necessary, $\fn^N\eta\subseteq (x_1,\ldots,x_i,\underline{T})S$ and therefore $\eta+(x_1,\ldots,x_i,T)S$ represents an element of $H^i_\fn(S/(\underline{T}))$ in the above direct limit system. We are assuming $S$ is weakly $F$-nilpotent. Thus $S/(\underline{T})$ is weakly $F$-nilpotent relative to $S$ by Theorem~\ref{Theorem criterion of F-nilpotence}. Therefore for all $e\gg0$ there exists a $t\in \mathbb{N}$ such that $x^{tp^e}\eta\in (x^{(t+1)p^e}_1,\ldots,x^{(t+1)p^e}_i, \underline{T}^{p^e})S$. By Lemma~\ref{lemma faithfully flat CM fiber} $(x^{(t+1)p^e}_1,\ldots,x^{(t+1)p^e}_i,\underline{T}^{p^e})S\cap R=(x^{(t+1)p^e}_1,\ldots,x^{(t+1)p^e}_i)R$. Therefore for all $e\gg0$ there exists $t\in \mathbb{N}$ such that $x^{tp^e}\eta\in (x^{(t+1)p^e}_1,\ldots,x^{(t+1)p^e}_i)R$, i.e., for all $e\gg 0$ the $e$th Frobenius action of $H^i_\fm(R)$ maps the element represented by $\eta+(x_1,\ldots,x_i)$ to $0$.

We now assume further that $S$ is $F$-nilpotent. To verify $R$ is $F$-nilpotent we may pass to the completion of $R$ and $S$ and assume $R$ is reduced, see \cite[Proposition~2.8]{PolstraQuy}. We aim to show that $0^F_{H^d_\fm(R)}=0^*_{H^d_\fm(R)}$. Let $x_1,\ldots, x_d$ be a system of parameters of $R$ and $\eta+(x_1,\ldots,x_d)$ represent an element of $0^*_{H^d_\fm(R)}$ under the direct limit identification
\[
H^d_\fm(R)\cong \varinjlim\left(\frac{R}{(x_1^t,\ldots,x_d^t)}\xrightarrow{\cdot x}\frac{R}{(x_1^{t+1},\ldots,x_d^{t+1})}\right).
\] 
Then there exists $c\in R^\circ$ with the property that for all $e\in\mathbb{N}$ there exists $t\in \mathbb{N}$ so that $c\eta^{p^e}x^{tp^e}\in (x_1^{(t+1)p^e},\ldots,x_d^{(t+1)p^e})$. By Theorem~\ref{Theorem criterion of F-nilpotence} we have that $0^{F_S}_{H^d_\fn(S/(\underline{T}))}=0^{*_S}_{H^d_\fn(S/(\underline{T}))}$. The module $0^{*_S}_{H^d_\fn(S/(\underline{T}))}$ is the tight closure of $0$ in $H^d_\fm(S/(\underline{T}))$ as an $S$-module and $0^{F_S}_{H^d_\fn(S/(\underline{T}))}$ is the Frobenius closure of $0$ in $H^d_\fm(S/(\underline{T}))$ as an $S$-module. Moreover, observe that $c\in S^\circ$. Therefore we find that $\eta+(x_1,\ldots,x_d,\underline{T})S$ represents an element of the tight closure of $0$ in $H^d_\fm(S/(\underline{T}))$ as an $S$-module. Since $S$ is $F$-nilpotent we have that for all $e\gg 0$ there exists a $t\in \mathbb{N}$ so that $\eta^{p^e}x^{tp^e}\in (x_1^{(t+1)p^e},\ldots,x_d^{(t+1)p^e},T^{p^e})S$. By (2) of Lemma~\ref{lemma faithfully flat CM fiber} we find that for all $e\gg0$ there exists $t\in \mathbb{N}$ so that $\eta^{p^e}x^{tp^e}\in (x_1^{(t+1)p^e},\ldots,x_d^{(t+1)p^e})$, i.e., $\eta+(x_1,\ldots, x_d)$ represents the class of an element of $0^F_{H^d_{\fm}(R)}$.
\end{proof}

The ascent results of Theorem~\ref{Main Theorem Ascent} are broken into several statements. We first show the property of being (weakly) $F$-nilpotent ascends under faithfully flat maps with Cohen-Macaulay fibers.

\begin{theorem}
\label{Theorem F-nilpotent acends}
Let $(R,\fm)\to (S,\fn)$ be a faithfully flat local ring map of prime characteristic $p>0$ rings with Cohen-Macaulay closed fiber $S/\fm S$.
\begin{enumerate}
\item If $R$ is weakly $F$-nilpotent then $S$ is weakly $F$-nilpotent.
\item Suppose additionally that $R\to S$ has geometrically regular fibers and $R/\sqrt{0}$ and $S\otimes_R R/\sqrt{0}$ have a test element in common. If $R$ is $F$-nilpotent then $S$ is $F$-nilpotent.
\end{enumerate} 
\end{theorem}

\begin{proof}
As in Lemma~\ref{lemma faithfully flat CM fiber} and Theorem~\ref{theorem F-nilpotent descends} we begin by choosing elements $T_1,\ldots, T_\ell$ in $S$ which are regular in $S$ and $S/\fm S$. Suppose first that $R$ is weakly $F$-nilpotent. Utilizing the notation of Notation~\ref{Notation T's} and the criteria of Theorem~\ref{Theorem criterion of F-nilpotence} the ring $S$ is weakly $F$-nilpotent if and only if $S/(\underline{T}^N)$ is weakly $F$-nilpotent relative to $S$ for each $N\in \mathbb{N}$.

Recall that the Frobenius action $S\to F^e_*S$ factors as $S\to F^e_*R\otimes_R S\xrightarrow{F^e_{S/R}} F^e_*S$ where $F^e_{S/R}$ is the relative Frobenius map and $S\to F^e_*R\otimes_R S$ is the base change of the Frobenius map of $R$ to $S$, see Subsection~\ref{Subsection relative Frobenius}. Tensoring with $S/(\underline{T}^N)$ we find that the relative Frobenius map $S/(\underline{T}^N)\to F^e_*S/(\underline{T}^{Np^e})$ has a factorization
\[
S/(\underline{T}^N)\to F^e_*R\otimes_R S/(\underline{T}^N)\xrightarrow{F^e_{S/R}} F^e_*S/(\underline{T}^{Np^e})
\]
In particular, Frobenius actions of local cohomology modules of $S$ with support in $\fn$ has a factorization 
\[
H^i_\fn(S/(\underline{T}^N))\to H^i_\fn(F^e_*R\otimes_R S/(\underline{T}^N))\xrightarrow{F^e_{S/R}} H^i_\fn(F^e_*S/(\underline{T}^{Np^e})).
\]
By Lemma~\ref{lemma faithfully flat CM fiber} the map $R\to S/(\underline{T}^N)$ is faithfully flat with $0$-dimensional fiber. Therefore the map $H^i_\fn(S/(\underline{T}^N))\to H^i_\fn(F^e_*R\otimes_R S/(\underline{T}^N))$ may be identified with
\[
H^i_\fm(R)\otimes_R S/(\underline{T}^N) \to H^i_\fm(F^e_*R)\otimes_R S/(\underline{T}^N).
\]
In other words, the relative Frobenius actions of local cohomology of $S/(\underline{T}^N)$ factor through the base change of the Frobenius actions of local cohomology of $R$. Therefore if $R$ is weakly $F$-nilpotent, i.e. all high enough iterates of Frobenius actions on lower local cohomology are $0$, then $S$ is also weakly $F$-nilpotent.

We now move to (2). Without loss of generality we may assume $R$ is reduced and therefore by assumption $R$ and $S$ have a test element $c$ in common. Moreover, we are assuming $R\to S$ has geometrically regular fibers, equivalently the relative Frobenius map $F^e_{S/R}:F^e_*R\otimes_R S\to F^e_*S$ is faithfully flat for all $e\in \mathbb{N}$ by Theorem~\ref{Theorem Radu-Andre}. The ring $S$ is $F$-nilpotent if and only if $S/(\underline{T}^N)$ is $F$-nilpotent relative to $S$ for all $N\in \mathbb{N}$ by Theorem~\ref{Theorem criterion of F-nilpotence}. In particular, by $(1)$ of this theorem it remains to show that $0^{F_S}_{H^d_\fm(S/(\underline{T}^N))}=0^{*_S}_{H^d_\fm(S/(\underline{T}^N))}$ for all $N\in \mathbb{N}$.

Because $c\in R$ the composition of $S$-linear maps
\[
S\to F^e_*S\xrightarrow{\cdot F^e_*c} F^e_*S
\]
can be factored as
\[
S\to F^e_*R\otimes_R S\xrightarrow{\cdot F^e_*c\otimes_R S}F^e_*R\otimes _R S\xrightarrow{F^e_{S/R}}F^e_*S
\]
where $S\to F^e_*R \otimes_R S$ is the base change of the Frobenius map of $R$. Base changing by $S/(\underline{T}^N)$ and examining the induced map of the top local cohomology modules we find that
\begin{equation}
\label{Equation factor this map}
H^d_\fn(S/(\underline{T}^N))\xrightarrow{F^e_S} H^d_\fn(F^e_*S/(\underline{T}^{Np^e}))\xrightarrow{\cdot F^e_*c} H^d_\fn (F^e_*S/(\underline{T}^{Np^e}))
\end{equation}
factors as 
% \begin{equation}
%     \begin{tikzcd}[column sep=small]
%     H^d_\fn(S/(\underline{T}^N))\arrow[r]&H^d_\fm(F^e_*R)\otimes_R S/(\underline{T}^N)\arrow[r]&H^d_\fm(F^e_*R)\otimes_R S/(\underline{T}^N)\arrow[r]&H^d_\fn(F^e_S/(T^{Np^e}))
%     \end{tikzcd}
% \end{equation}
\begin{align}
\begin{split}
\label{Equation this is the factored map}
H^d_\fn(S/(\underline{T}^{N}))\xrightarrow{F^e_R\otimes_R S/(\underline{T}^N)}& H^d_\fm(F^e_*R)\otimes_R S/(\underline{T}^N)\\
&\xrightarrow{\cdot F^e_*c\otimes _R S/(\underline{T}^N)}H^d_\fm(F^e_*R)\otimes_R S/(\underline{T}^N)\xrightarrow{F^e_{S/R}\otimes_R S/(\underline{T}^N)} H^d_\fn(F^e_S/(T^{Np^e})).
\end{split}
\end{align}

We aim to show that if $\eta\in H^d_\fn(S/(\underline{T}^N))$ is an element of the kernel of the composition of maps in \ref{Equation factor this map} for all $e\gg 0$ then $\eta$ is an element of the kernel of $F^e_S$ for all $e\gg 0$. Given such an element $\eta$  we find that $\eta$ is an element of the kernel of $(\cdot F^e_*c\otimes_R S/(\underline{T}^N))\circ (F^e_R\otimes_R S/(\underline{T}^N))$ from \ref{Equation this is the factored map} for all $e\gg 0$ since $F^e_{S/R}$ is faithfully flat. But $R\to S/(\underline{T}^N)$ is faithfully flat and we are assuming $R$ is $F$-nilpotent, i.e. the kernel of $\cdot F^e_*c\circ F^e_R$ agrees with the kernel of $F^e_R$ for all $e\gg 0$. Therefore $\eta $ is an element of the kernel of $F^e_R\otimes_R S/(\underline{T}^N)$ for all $N\gg 0$. It follows that $\eta$ is an element of the kernel of $F^e_S$ for all $e\gg 0$ since $F^e_S$ factors as $F^e_{S/R}\circ (F^e_*R\otimes_R S)$.
\end{proof}

\begin{theorem}
\label{faithfully flat purely inseparable}
Let $(R,\fm)\to (S,\fn)$ be a purely inseparable local homomorphism of prime characteristic $p>0$ rings.
\begin{enumerate}
\item If $R$ is weakly $F$-nilpotent then $S$ is weakly $F$-nilpotent.
\item If $R\to S$ is faithfully flat and $R$ is $F$-nilpotent then $S$ is $F$-nilpotent.
\end{enumerate}  
\end{theorem}

\begin{proof}
Because $R\to S$ is purely inseparable the induced map of spectra is a bijection. In particular, $R$ and $S$ have the same Krull dimension. Let $d=\dim (R)=\dim(S)$, suppose $R$ is weakly $F$-nilpotent, and let $0\leq i <d$. Let $x_1,\ldots, x_i$ be a filter regular sequence of $R$. We first observe that $x_1,\ldots, x_i$ is also a filter regular sequence of $S$. Indeed, if $x_{i+1}$ avoids all non-maximal associated primes of $(x_1,\ldots, x_{i})$ in $R$ then $x_{i+1}$ avoids all non-maximal associated primes of $(x_1,\ldots,x_{i})S$ since every prime ideal of $S$ is the radical of a prime ideal extended from $R$. We may therefore identify the local cohomology modules $H^i_\fm(R)$ and $H^i_\fn(S)$ as in Theorem~\ref{Nagel-Schenzel} with respect to the sequence $x_1,\ldots, x_i$.

 Suppose $\eta+(x_1,\ldots,x_i)S$ is a representative of an element of $H^i_\fn(S)$. Because $R\to S$ is purely inseperable there exists an $e_0\in \mathbb{N}$ so that $\eta^{p^{e_0}}\in R$. Thus, after replacing $x_1,\ldots,x_i$ by $x_1^{p^{e_0}},\ldots, x_i^{p^{e_0}}$ and $\eta+(x_1,\ldots,x_i)S$ by its image under the $e_0$-iterate of the Frobenius action on $H^i_\fn(S)$, we may assume $\eta\in R$. In particular, $\eta+(x_1,\ldots,x_i)R$ represents an element of $H^i_\fm(R)$ and therefore, because $R$ is weakly $F$-nilpotent, for all $e\gg0$ there exists $t\in \mathbb{N}$ so that $\eta^{p^e}x^{tp^e}\in (x_1^{(t+1)p^e},\ldots,x_i^{(t+1)p^e})R$ where $x=x_1\cdots x_i$. Passing to $S$ we see the same containment holds in the parameter ideal extended to $S$, i.e., the $e$th Frobenius action of $H^i_\fn(S)$ maps $\eta+(x_1,\ldots,x_i)$ to $0$.

Suppose even further that $R$ is $F$-nilpotent and $R\to S$ is faithfully flat. Similar to before, to check that $S$ is $F$-nilpotent we may assume $R$ and $S$ are complete and that $R$ is reduced. Let $c\in R^\circ$ be a common test element of $R$ and $S$, $x_1,\ldots,x_d$ a common system of parameters of $R$ and $S$, and $\eta+(x_1,\ldots,x_d)S$ a representative of an element of $0^*_{H^d_\fn(S)}$. Replacing $x_1,\ldots,x_d$ by $x_1^{p^{e_0}},\ldots,x_d^{p^{e_0}}$ and $\eta+(x_1,\ldots, x_d)$ by its image under the $e_0$-iterate of the Frobenius action on $H^d_\fm(R)$ we may assume $\eta\in S$. Because $\eta+(x_1,\ldots, x_d)\in 0^*_{H^d_\fn(S)}$ we have that for all $e\in \mathbb{N}$ there exists $t\in \mathbb{N}$ so that $c\eta^{p^e}x^{tp^e}\in (x_1^{(t+1)p^e},\ldots,x_d^{(t+1)p^e})S$. But $R\to S$ is faithfully flat and therefore for all $e\gg 0$ there exists $t\in \mathbb{N}$ so that $c\eta^{p^e}x^{tp^e}\in (x_1^{(t+1)p^e},\ldots,x_d^{(t+1)p^e})R$. Hence $\eta+(x_1,\ldots,x_d)R$ represents an element of $0^*_{H^d_\fm(R)}=0^F_{H^d_\fm(R)}$. Hence for all $e\gg 0$ there exists $t\in \mathbb{N}$ so that $\eta^{p^e}x^{tp^e}\in (x_1^{(t+1)p^e},\ldots,x_d^{(t+1)p^e})R$. Passing to $S$ we find that $\eta+(x_1,\ldots,x_d)$ represents an element of $0^F_{H^d_\fn(S)}$.
\end{proof}

We next prove faithfully flat ascent and descent results for generalized weakly $F$-nilpotent rings. We accomplish this by first realizing generalized weakly $F$-nilpotent rings as rings which are weakly $F$-nilpotent on the punctured spectrum.

\begin{proposition}
\label{proposition gwfn = wfn on punctured spectrum}
Suppose $(R,\fm,k)$ is a local excellent ring of dimension $d$ and of prime characteristic $p>0$. Then $R$ is generalized weakly $F$-nilpotent if and only if $R_\fp$ is weakly $F$-nilpotent for all $\fp \in \Spec(R)\setminus\{\fm\}$. 
\end{proposition}

\begin{proof}
First suppose that $R$ is $F$-finite. Let $\omega_R^\bullet$ be a normalized dualizing complex of $R$. The ring $R$ is generalized weakly $F$-nilpotent if and only if $H^i_\fm(R)/0^F_{H^i_\fm(R)}$ has finite length for all $0\leq i< d$. By Proposition~\ref{Proposition Frobenius is just Grothendieck Trace}, this is equivalent to the image of the evaluation at $1$ map being finite for each $0\leq i < d$:
\[
(F^e)^\vee(-i): F^e_*H^{-i}(\omega_R^\bullet)\to H^{-i}(\omega_R^\bullet).
\]
Therefore $R$ is generalized weakly $F$-nilpotent if and only if for all $0\leq i<d$ and $e\gg 0$ the image of $(F^e)^\vee(-i)$ is finite, i.e. $(F^e)^\vee(-i)$ is the $0$-map on the punctured spectrum, i.e. $R$ is weakly $F$-nilpotent on the punctured spectrum.

Now suppose $R$ is an excellent local ring. Then $R\to \widehat{R}$ has geometrically regular fibers. By Theorem~\ref{Theorem F-nilpotent acends} we may assume $R$ is complete. In which case, there exists a faithfully flat and purely inseparable map $R\to R^\Gamma$ to some $F$-finite local domain $R^\Gamma$ (see \cite[Section~6]{HochsterHunekeTransactions}). The proposition now follows by Theorem~\ref{Theorem F-nilpotent acends} and the previous paragraph.
\end{proof}

Proposition~\ref{proposition gwfn = wfn on punctured spectrum} allows us to understand the property of being generalized weakly $F$-nilpotent under faithfully flat local ring maps.

\begin{theorem}
\label{theorem generalized weakly F-nilpotent faithfully flat map}
Let $(R,\fm)\to (S,\fn)$ be a faithfully flat map of local rings of prime characteristic $p>0$ with $0$-dimensional fiber.
\begin{enumerate}
\item If $S$ is generalized weakly $F$-nilpotent then $R$ is generalized weakly $F$-nilpotent.
\item If $R$ is generalized weakly $F$-nilpotent then $S$ is generalized weakly $F$-nilpotent provided one of the following conditions is satisfied:
\begin{itemize}
\item $S/\fm S$ is a field, and $R/\sqrt{0}$ and $S\otimes_R R/\sqrt{0}$ have a test element in common;
\item $R\to S$ is purely inseparable.
\end{itemize}
\end{enumerate}
\end{theorem}

\begin{proof}
Without loss of generality we may assume $R$ and $S$ are complete. The theorem readily follows by Proposition~\ref{proposition gwfn = wfn on punctured spectrum}, Theorem~\ref{theorem F-nilpotent descends}, Theorem~\ref{Theorem F-nilpotent acends}, and Theorem~\ref{faithfully flat purely inseparable}.
\end{proof}

Recall that generalized weakly $F$-nilpotent rings were introduced by Maddox in \cite{Maddox} for the purpose of finding a new class of local rings which have finite Frobenius test exponent. Prior to \cite{Maddox}, the following classes of rings were shown to have finite Frobenius test exponents:
\begin{enumerate}
\item Cohen-Macaulay rings, \cite{KatzmanSharp};
\item generalized Cohen-Macaulay rings, \cite{HunekeKatzmanSharpYao};
\item weakly $F$-nilpotent rings, \cite{Quy}.
\end{enumerate}
We provide an example of a generalized weakly $F$-nilpotent ring which is neither generalized Cohen-Macaulay nor weakly $F$-nilpotent. Therefore the main results of \cite{Maddox} were not already covered by the results of \cite{HunekeKatzmanSharpYao, KatzmanSharp, Quy}.

\begin{example}\label{gwfn-example}
Let $(A,\fm,k)$ be a weakly $F$-nilpotent local domain of dimension $2$ and depth $1$, essentially of finite type over a perfect field $K$ of prime characteristic $p$. Let $B' = A[T]$ be a polynomial ring in one variable over $A$ and $\fn = (\underline{x},T)$. Define $(B,\mathfrak{n}) =(B'_{\mathfrak{n}},\mathfrak{n})$. Then, there is a ring $R$ with $A \subset R \subset B$ such that the ideal $\mathfrak{c}=\Ann_R(B/R)$ is a maximal ideal of $R$. Further, $R_\mathfrak{c}$ is neither generalized Cohen-Macaulay nor weakly $F$-nilpotent, but is generalized weakly $F$-nilpotent.
\end{example}

\begin{proof}[Proof of Example~\ref{gwfn-example}] Present $A$ as $A=(K[x_1,\cdots,x_n]/I)_\fp$ for some $\fp \in \Spec(K[x_1,\cdots,x_n])$, and it suffices to assume $\fp = (x_1,\cdots,x_n)=(\underline{x})$, so $\fm = (\underline{x})$.

The map $A\to B$ is faithfully flat with regular closed fiber. Therefore the ring $(B,\fn)$ is weakly $F$-nilpotent of dimension 3 and depth 2 by Theorem~\ref{Theorem F-nilpotent acends}. The ideal $\fm B'\subsetneq \fn$ is prime (as $B'/\fm B' \cong K[T]$ is a domain) and $(B')_{\fm B'}$ is not Cohen-Macaulay as $A$ is not Cohen-Macaulay, so $B'$ is not generalized Cohen-Macaulay. Consequently, $B$ is not generalized Cohen-Macaulay and so we must have $H^2_\fn(B)$ is not of finite length.

We let \[ 
R' = \{f(x_1,\cdots, x_n, T) \in B' \mid f(0,\cdots,0,0)=f(0,\cdots,0,1) \}.
\] 
Consider the ideal $\mathfrak{c} = \Ann_{R'}(B'/R')$, an ideal common to both $R'$ and $B'$. One observes that $\mathfrak{c} B' = (\underline{x},T(T-1)) \subset \fn$ and $\mathfrak{c}$ is a maximal ideal of $R'$. Set $(R,\mathfrak{c}) =(R'_\mathfrak{c},\mathfrak{c})$ so that there are local inclusions
\[ 
(A,\fm) \subset (R,\mathfrak{c}) \subset (B,\fn).
\]
Note that $T-1$ is a unit in $B$ so $\mathfrak{c} B = \fn$. Moreover, $R$ is of dimension 3 since $R\subset B$ is module-finite.

There is a short exact sequence of $R$-modules: 
\begin{center}
\begin{tikzcd}
0 \arrow{r} & R \arrow{r} & B \arrow{r} & B/R \arrow{r} & 0
\end{tikzcd}
\end{center} 
and therefore there is a long exact sequence 
\begin{center}
\begin{tikzcd}
0 \arrow{r} & B/R \arrow{r} & H^1_\mathfrak{c}(R) \arrow{r} & 0 \arrow{r} & 0 \arrow{r} & H^2_\mathfrak{c}(R) \arrow{r} & H^2_\fn(B) \arrow{r} & 0
\end{tikzcd}.
\end{center}

The Frobenius action on $B/R$ is given by $b+R \mapsto b^{p} + R$, and $T^{p^e} +R \neq R$ for all $e \in \mathbb{N}$, so $H^1_\mathfrak{c}(R) \cong B/R$ is not $F$-nilpotent but is finite length (since $\Ann(B/R)=\mathfrak{c}$). Furthermore $H^2_\mathfrak{c}(R) \cong H^2_\fn(B)$ is nilpotent but is not finite length. 
\end{proof}

\section{Open loci results}\label{Section open results}

Theorem~\ref{Open loci results} is a result likely already understood by experts in the $F$-finite scenario, see \cite[Lemma~2.3]{SrinivasTakagi}. Nevertheless, we present a complete proof of Theorem~\ref{Open loci results} in the $F$-finite case for sake of completeness, as well as the essentially of finite type over an excellent local ring scenario. We begin with a consequence of Proposition~\ref{Proposition Frobenius is just Grothendieck Trace}.

\begin{lemma}
\label{Lemma Frobenius closure of the 0 module}
Let $R$ be a locally equidimensional $F$-finite ring of prime characteristic $p>0$ and let $\omega_R^\bullet$ be a normalized dualizing complex of $R$. There exists a quotient $N$ of $\omega_R:=H^{-d}(\omega_R^\bullet)$ so that for each $\fp\in \Spec(R)$ the $R_\fp$-Matlis dual of $N_\fp$ is $0^F_{H^{\Height(\fp)}_{\fp R_\fp}(R_\fp)}$. 
\end{lemma}

\begin{proof}
By Proposition~\ref{Proposition Frobenius is just Grothendieck Trace} there exists an $e$ so that $0^F_{H^d_{\fp R_\fp}(R_\fp)}$ is the $R_\fp$-Matlis dual of the localized cokernel of the Cartier linear map
\[
F^e_*H^{-d}(\omega_R^\bullet)\xrightarrow{(F^e)^\vee} H^{-d}(\omega_R^\bullet).
\]
\end{proof}

\begin{lemma}
\label{Lemma tight closure of the 0 module}
Let $R$ be a locally equidimensional $F$-finite ring of prime characteristic $p>0$ and let $\omega_R^\bullet$ be a normalized dualizing complex of $R$. Suppose further that $R$ admits a completely stable test element\footnote{A completely stable test element is an element $c\in R$ which serves as a test element of $R$, every localization of $R$, and every completion of $R$ with respect to an ideal.}, e.g. $R$ is reduced\footnote{See \cite[Theorem~5.10]{HochsterHunekeTransactions} for an explanation of why $F$-finite reduced rings admit a completely stable test element.}. Then there exists a quotient $M$ of $\omega_R:=H^{-d}(\omega_R^\bullet)$ so that for each $\fp\in \Spec(R)$ the $R_\fp$-Matlis dual of $M$ is $0^*_{H^{\Height(\fp)}_{\fp R_\fp}(R_\fp)}$. 
\end{lemma}

\begin{proof}
Let $c\in R^\circ$ be a completely stable test element and $e_0$ a test exponent for $0^*_{H_{\fp R_\fp}(R_\fp)}$ with respect to $c$ for all $\fp \in \Spec(R)$ (see Corollary~\ref{Finite test exponents for tight closure}). Then $0^*_{H^{\Height(\fp)}_{\fp R_\fp}(R_\fp)}$ is realized as the kernel of the following composition of maps:
\[
H^{\Height(\fp)}_{\fp R_\fp}(R_\fp)\xrightarrow{F^{e_0}}F^e_*H^{\Height(\fp)}_{\fp R_\fp}(R_\fp)\xrightarrow{\cdot F^e_*c} F^e_*H^{\Height(\fp)}_{\fp R_\fp}(R_\fp).
\]
By Proposition~\ref{Proposition Frobenius is just Grothendieck Trace} the kernel of the above composition of maps is the $R_\fp$-Matlis dual of the localization at $\fp$ of the cokernel of the composition
\[
F^e_*H^{-d}(\omega_R^\bullet)\xrightarrow{\cdot F^e_*c} F^e_*H^{-d}(\omega_R^\bullet)\xrightarrow{(F^e)^\vee}H^{-d}(\omega_R^\bullet).
\]
\end{proof}

The above lemmas and the results of Section~\ref{Section flat base change results} provide us with a proof of Theorem~\ref{Open loci results}.

\begin{theorem}
\label{Theorem Open}
Let $R$ be a ring of prime characteristic $p>0$ which is either $F$-finite or essentially of finite type over an excellent local ring. Then
\begin{enumerate}
\item $\{\fp\in \Spec(R)\mid R_\fp\mbox{ is weakly $F$-nilpotent}\}$ is an open subset of $\Spec(R)$;
\item $\{\fp\in \Spec(R)\mid R_\fp\mbox{ is $F$-nilpotent}\}$ is an open subset $\Spec(R)$.
\end{enumerate} 
\end{theorem}

\begin{proof}
Suppose first that $R$ is $F$-finite. We may assume $R$ is a reduced ring by \cite[Proposition~2.8 (2)]{PolstraQuy}. Denote by $\omega_R^\bullet$ a dualizing complex of $R$. Let $\fp\in \Spec(R)$ and suppose that $R_\fp$ is weakly $F$-nilpotent. By \cite[Proposition~2.8 (3) and Remark~2.9]{PolstraQuy}, $R_\fp$ is an equidimensional local ring. Thus, by replacing $R$ by a suitable localization preserving $\fp$, we may assume $R$ is locally equidimensional (see \cite[Proof of Corollary~5.3]{PolstraTransactions}). Suppose that $R$ has Krull dimension $d$. We shift the grading of $\omega_R^\bullet$ so that $\omega_R^\bullet$ is normalized. Then $R_\fp$ is weakly $F$-nilpotent if and only if for all $i<\Height(\fp)$ the Cartier linear map $ (F^e_\fp)^\vee(-i)=0$ for all $e\gg 0$. Equivalently, $\fp$ is in the weakly $F$-nilpotent locus of $R$ if and only if for all $i>0$ one has that $(F^e)^\vee(-d+i)$ localizes to the $0$-map at $\fp$ for all $e\gg 0$. The vanishing locus of a map of finitely generated modules is an open set and therefore the weakly $F$-nilpotent locus of $R$ is indeed open.

Now suppose $\fp\in \Spec(R)$ and $R_\fp$ is $F$-nilpotent. As before we may assume $R$ is reduced and locally equidimensional. Let $N$ be as in Lemma~\ref{Lemma Frobenius closure of the 0 module} and $M$ as in Lemma~\ref{Lemma tight closure of the 0 module}. Then $N\subseteq M$ and for each $\fp$ the $R_\fp$-Matlis dual of $(M/N)_\fp$ is $0^*_{H^{\Height(\fp)}_{\fp R_\fp}(R_\fp)}/0^F_{H^{\Height(\fp)}_{\fp R_\fp}(R_\fp)}$. The collection of primes $\{\fp\in \Spec(R)\mid (M/N)_\fp=0\}$ is open, the intersection of this open set with the weakly $F$-nilpotent locus is open, and this open set defines the $F$-nilpotent locus.

Now suppose that $R$ is essentially of finite type over an excellent local ring $A$. If $\widehat{A}$ is the completion of $A$ at its maximal ideal then $R\to \widehat{A}\otimes_A R$ is faithfully flat with geometrically regular fibers. The induced map $\Spec(\widehat{A}\otimes_A R)\rightarrow\Spec R$ is open, hence if the (weakly) $F$-nilpotent locus of $\Spec(R\otimes_A\widehat{A})$ is open then so is the (weakly) $F$-nilpotent locus of $\Spec(R)$. By Theorem~\ref{Theorem F-nilpotent acends} we may assume $R$ is essentially of finite type over a complete local ring. Hence there exists a faithfully flat and purely inseparable map $R\to R^\Gamma$ where $R^\Gamma$ is an $F$-finite ring (see \cite[Section~6]{HochsterHunekeTransactions}). Therefore the subset of $\Spec(R)$ of (weakly) $F$-nilpotent primes is indeed open by Theorem~\ref{faithfully flat purely inseparable} and the $F$-finite case of the theorem.
 \end{proof}

\begin{remark}
\begin{enumerate}
    \item One should not expect the conclusions of Theorem \ref{Theorem Open} to hold under more general hypotheses than those stated. Indeed, \cite[Proposition 2]{hochster1973} and faithfully flat descent for $F$-nilpotent (resp. weakly $F$-nilpotent) rings guarantees the existence of a locally excellent local Noetherian ring whose $F$-nilpotent (resp. weakly $F$-nilpotent) locus is not open. See \cite[Example 5.10]{DattaMurayama} for more details.
    \item In light of Proposition~\ref{proposition gwfn = wfn on punctured spectrum} and Example~\ref{gwfn-example}, one should not expect the generalized weakly $F$-nilpotent locus of a local ring of prime characteristic to be open in general.
\end{enumerate}
\end{remark}

We conclude this paper by showing that the property of being (weakly) $F$-nilpotent descends under arbitrary faithfully flat maps, thus concluding the proof of Theorem~\ref{Main Theorem Ascent}.

\begin{theorem}
\label{theorem F-nilpotent descends-non-CM}
Let $(R,\fm)\to (S,\fn)$ be a faithfully flat map of local rings of prime characteristic $p>0$. 
\begin{enumerate}
    \item If $S$ is weakly $F$-nilpotent then so is $R$.
    \item If $S$ is excellent and $F$-nilpotent then $R$ is $F$-nilpotent.
\end{enumerate}
\end{theorem}
\begin{proof}
The property of being weak $F$-nilpotent can be checked after completion. Therefore in statement (1) of the theorem we may assume $S$ is an excellent local ring. Let $\fp\in\Spec(S)$ be a prime ideal minimal over $\fm S$. If $S$ is either weakly $F$-nilpotent or $F$-nilpotent then so is $S_\fp$ respectively by Theorem~\ref{Theorem Open} with the stated assumptions on $S$ (cf \cite[Corollary~5.17]{PolstraQuy}). Replacing $S$ by $S_\fp$ we may assume $R\to S$ is faithfully flat with $0$-dimensional closed fiber. The result now follows by Theorem~\ref{theorem F-nilpotent descends}.
\end{proof}

\section{Acknowledgments}
We are grateful to Ian Aberbach, Takumi Murayama and Pham Hung Quy for comments on earlier drafts of this article. We thank Rankeya Datta for insightful conversations. The fourth named author is grateful to his advisor, Kevin Tucker, for his constant encouragement. We also thank the anonymous referee for valuable feedback on a previous version of this article.

\bibliographystyle{alpha}
\bibliography{References}

\newcommand{\etalchar}[1]{$^{#1}$}
\begin{thebibliography}{KMVZ17}

\bibitem[And93]{AndreRelative}
Michel Andr\'{e}.
\newblock Homomorphismes r\'{e}guliers en caract\'{e}ristique {$p$}.
\newblock {\em C. R. Acad. Sci. Paris S\'{e}r. I Math.}, 316(7):643--646, 1993.

\bibitem[BB05]{BlickleBondu}
Manuel Blickle and Raphael Bondu.
\newblock Local cohomology multiplicities in terms of \'{e}tale cohomology.
\newblock {\em Ann. Inst. Fourier (Grenoble)}, 55(7):2239--2256, 2005.

\bibitem[BB11]{BlickleBockle}
Manuel Blickle and Gebhard B\"{o}ckle.
\newblock Cartier modules: finiteness results.
\newblock {\em J. Reine Angew. Math.}, 661:85--123, 2011.

\bibitem[Bre06]{Brenner}
Holger Brenner.
\newblock Bounds for test exponents.
\newblock {\em Compos. Math.}, 142(2):451--463, 2006.

\bibitem[DM19]{DattaMurayama}
Rankeya Datta and Takumi Murayama.
\newblock Permanence properties of {$F$}-injectivity, 2019.

\bibitem[EH08]{EnescuHochster}
Florian Enescu and Melvin Hochster.
\newblock The {F}robenius structure of local cohomology.
\newblock {\em Algebra Number Theory}, 2(7):721--754, 2008.

\bibitem[Gab04]{Gabber}
Ofer Gabber.
\newblock Notes on some {$t$}-structures.
\newblock In {\em Geometric aspects of {D}work theory. {V}ol. {I}, {II}}, pages
  711--734. Walter de Gruyter, Berlin, 2004.

\bibitem[Har66]{HartshorneResiduesAndDuality}
Robin Hartshorne.
\newblock {\em Residues and duality}.
\newblock Lecture notes of a seminar on the work of A. Grothendieck, given at
  Harvard 1963/64. With an appendix by P. Deligne. Lecture Notes in
  Mathematics, No. 20. Springer-Verlag, Berlin-New York, 1966.

\bibitem[HH90]{HochsterHunekeJAMS}
Melvin Hochster and Craig Huneke.
\newblock Tight closure, invariant theory, and the {B}rian\c{c}on-{S}koda
  theorem.
\newblock {\em J. Amer. Math. Soc.}, 3(1):31--116, 1990.

\bibitem[HH94]{HochsterHunekeTransactions}
Melvin Hochster and Craig Huneke.
\newblock {$F$}-regularity, test elements, and smooth base change.
\newblock {\em Trans. Amer. Math. Soc.}, 346(1):1--62, 1994.

\bibitem[HKSY06]{HunekeKatzmanSharpYao}
Craig Huneke, Mordechai Katzman, Rodney~Y. Sharp, and Yongwei Yao.
\newblock Frobenius test exponents for parameter ideals in generalized
  {C}ohen-{M}acaulay local rings.
\newblock {\em J. Algebra}, 305(1):516--539, 2006.

\bibitem[Hoc73]{hochster1973}
M.~Hochster.
\newblock Non-openness of loci in noetherian rings.
\newblock {\em Duke Math. J.}, 40(1):215--219, 03 1973.

\bibitem[HQ19]{HuongQuy}
Duong~Thi Huong and Pham~Hung Quy.
\newblock Notes on the {F}robenius test exponents.
\newblock {\em Comm. Algebra}, 47(7):2702--2710, 2019.

\bibitem[HS77]{HartshorneSpeiser}
Robin Hartshorne and Robert Speiser.
\newblock Local cohomological dimension in characteristic {$p$}.
\newblock {\em Ann. of Math. (2)}, 105(1):45--79, 1977.

\bibitem[ILL{\etalchar{+}}07]{24Hours}
Srikanth~B. Iyengar, Graham~J. Leuschke, Anton Leykin, Claudia Miller, Ezra
  Miller, Anurag~K. Singh, and Uli Walther.
\newblock {\em Twenty-four hours of local cohomology}, volume~87 of {\em
  Graduate Studies in Mathematics}.
\newblock American Mathematical Society, Providence, RI, 2007.

\bibitem[KMVZ17]{GlobalParameterTestIdeals}
Mordechai Katzman, Serena Murru, Juan~D. V{\'{e}}lez, and Wenliang Zhang.
\newblock Global parameter test ideals.
\newblock {\em J. Algebra}, 486:80--97, 2017.

\bibitem[KS06]{KatzmanSharp}
Mordechai Katzman and Rodney~Y. Sharp.
\newblock Uniform behaviour of the {F}robenius closures of ideals generated by
  regular sequences.
\newblock {\em J. Algebra}, 295(1):231--246, 2006.

\bibitem[Kun76]{Kunz1976}
Ernst Kunz.
\newblock On {N}oetherian rings of characteristic {$p$}.
\newblock {\em Amer. J. Math.}, 98(4):999--1013, 1976.

\bibitem[KZ19]{KatzmanZhang}
Mordechai Katzman and Wenliang Zhang.
\newblock Multiplicity bounds in prime characteristic.
\newblock {\em Comm. Algebra}, 47(6):2450--2456, 2019.

\bibitem[Lyu97]{LyubeznikFmodules}
Gennady Lyubeznik.
\newblock {$F$}-modules: applications to local cohomology and {$D$}-modules in
  characteristic {$p>0$}.
\newblock {\em J. Reine Angew. Math.}, 491:65--130, 1997.

\bibitem[Lyu06]{Lyubeznik}
Gennady Lyubeznik.
\newblock On the vanishing of local cohomology in characteristic {$p>0$}.
\newblock {\em Compos. Math.}, 142(1):207--221, 2006.

\bibitem[Mad19]{Maddox}
Kyle Maddox.
\newblock A sufficient condition for the finiteness of {F}robenius test
  exponents.
\newblock {\em Proc. Amer. Math. Soc.}, 147(12):5083--5092, 2019.

\bibitem[Mat86]{Matsumura}
Hideyuki Matsumura.
\newblock {\em Commutative ring theory}, volume~8 of {\em Cambridge Studies in
  Advanced Mathematics}.
\newblock Cambridge University Press, Cambridge, 1986.
\newblock Translated from the Japanese by M. Reid.

\bibitem[{Mur}18]{TakumiGamma}
Takumi {Murayama}.
\newblock {The gamma construction and asymptotic invariants of line bundles
  over arbitrary fields}.
\newblock {\em arXiv e-prints}, page arXiv:1809.01217, Sep 2018.

\bibitem[NS94]{NagelSchenzel}
Uwe Nagel and Peter Schenzel.
\newblock Cohomological annihilators and {C}astelnuovo-{M}umford regularity.
\newblock In {\em Commutative algebra: syzygies, multiplicities, and birational
  algebra ({S}outh {H}adley, {MA}, 1992)}, volume 159 of {\em Contemp. Math.},
  pages 307--328. Amer. Math. Soc., Providence, RI, 1994.

\bibitem[Pol18]{PolstraTransactions}
Thomas Polstra.
\newblock Uniform bounds in {F}-finite rings and lower semi-continuity of the
  {F}-signature.
\newblock {\em Trans. Amer. Math. Soc.}, 370(5):3147--3169, 2018.

\bibitem[PQ19]{PolstraQuy}
Thomas Polstra and Pham~Hung Quy.
\newblock Nilpotence of {F}robenius actions on local cohomology and {F}robenius
  closure of ideals.
\newblock {\em J. Algebra}, 529:196--225, 2019.

\bibitem[Quy19]{Quy}
Pham~Hung Quy.
\newblock On the uniform bound of {F}robenius test exponents.
\newblock {\em J. Algebra}, 518:119--128, 2019.

\bibitem[Rad92]{RaduRelative}
Nicolae Radu.
\newblock Une classe d'anneaux noeth\'{e}riens.
\newblock {\em Rev. Roumaine Math. Pures Appl.}, 37(1):79--82, 1992.

\bibitem[Sha06]{SharpTightClosureTestExponents}
Rodney~Y. Sharp.
\newblock Tight closure test exponents for certain parameter ideals.
\newblock {\em Michigan Math. J.}, 54(2):307--317, 2006.

\bibitem[Smi97]{SmithFrational}
Karen~E. Smith.
\newblock {$F$}-rational rings have rational singularities.
\newblock {\em Amer. J. Math.}, 119(1):159--180, 1997.

\bibitem[ST17]{SrinivasTakagi}
Vasudevan Srinivas and Shunsuke Takagi.
\newblock Nilpotence of {F}robenius action and the {H}odge filtration on local
  cohomology.
\newblock {\em Adv. Math.}, 305:456--478, 2017.

\bibitem[{Sta}18]{StacksProject}
The {Stacks Project Authors}.
\newblock \textit{Stacks Project}.
\newblock \url{https://stacks.math.columbia.edu}, 2018.

\bibitem[V{\'{e}}l95]{Velez}
Juan~D. V{\'{e}}lez.
\newblock Openness of the {$F$}-rational locus and smooth base change.
\newblock {\em Journal of Algebra}, 172(2):425 -- 453, 1995.

\end{thebibliography}

\end{document}